\documentclass{article}
\usepackage[utf8]{inputenc}
\usepackage{dsfont}
\usepackage{amsthm}
\usepackage{amsmath}
\usepackage{caption}
\usepackage{subcaption}
\usepackage{graphicx}
\usepackage{hyperref}
\usepackage[ruled,vlined]{algorithm2e}
\usepackage[longnamesfirst]{natbib}
\usepackage{pgfplots}
\usepackage[title]{appendix}
\usepackage{pgfplotstable} 
\usetikzlibrary{calc}
\pgfplotsset{width=6.5cm,compat=1.7}
\usepackage{rotating}
\usepackage{tabularx}
\usepackage{bm}
\DeclareMathOperator{\stat}{stat}
\usepackage{nicefrac}

\usepackage{xcolor}

\newtheorem{remark}{Remark}
\newtheorem{proposition}{Proposition}
\newtheorem{corollary}{Corollary}
\newtheorem{lemma}{Lemma}
\newtheorem{definition}{Definition}
\definecolor{color1}{HTML}{1f77b4}
\definecolor{color2}{HTML}{ff7f0e}
\definecolor{color3}{HTML}{2ca02c}
\definecolor{color4}{HTML}{d62728}
\definecolor{color5}{HTML}{9467bd}
\tikzset{every node/.prefix style={font=\footnotesize}}
\pgfplotsset{twoonpage/.style={height=200,width=300}}
   \pgfplotsset{proxystyle/.style={
      x tick label style={font=\tiny,rotate=90},
      grid = both,
      xtick = {0,  27,  54,  81, 108, 135, 162, 189, 216, 243},
      xticklabels={18-08 , 19-02, 19-08, 20-02, 20-09, 21-03, 21-09, 22-03,22-09, 23-03},
      xmin = 0,
      xmax = 270
   }} 

      \pgfplotsset{SimStyle/.style={
      x tick label style={font=\tiny,rotate=90},
      grid = both,
      xtick = { 0,  25,  50,  75, 100, 125, 150, 175, 200, 225, 250, 275, 300, 325, 350, 375, 400, 425, 450, 475,500},
      xticklabels={0,  1,  2,  3,  4,  5,  6,  7,  8,  9, 10, 11, 12, 13, 14, 15, 16,
       17, 18, 19,20},
      xmin = 0,
      xmax = 500
   }}

      \pgfplotsset{PriceModelStyle/.style={
      x tick label style={font=\tiny,rotate=90},
      grid = both,
      xtick = { 0,  25,  50,  75, 100, 125, 150, 175, 200, 225, 250, 275, 300, 325, 350, 375, 400, 425, 450, 475,500},
      xticklabels={0,  1,  2,  3,  4,  5,  6,  7,  8,  9, 10, 11, 12, 13, 14, 15, 16,
       17, 18, 19,20},
      xmin = 0,
      xmax = 500
   }}

 \pgfplotsset{autocorrelatiostyle/.style={
      x tick label style={font=\tiny,rotate=90},
      grid = both,
      xtick = { 0, 25,50,75,100,125,150},
      xticklabels={ 0, 25,50,75,100,125,150},
      xmin = 0,
      xmax = 150
   }}

 \pgfplotsset{AutoCorrelationStyle2/.style={
      x tick label style={font=\tiny,rotate=90},
      grid = both,
      xtick = { 0, 50,100,150,200,250,300},
      xticklabels={ 0, 50,100,150,200,250,300},
      xmin = 0,
      xmax = 300
   }}

\title{Stationary solution for a fractional stochastic delay differential equations with multiple delays}
\author{{\'A}lvaro Guinea Juliá and Alet Roux}

\begin{document}

\maketitle

\begin{abstract} This paper studies a linear stochastic delay differential equation driven by fractional Brownian motion and involving a finite number of discrete delays. The model combines two sources of memory: delayed feedback in the drift and temporal dependence induced by the Hurst parameter of the fractional Brownian motion. We first introduce the fundamental solution associated with the corresponding deterministic delay equation and use it to obtain an explicit representation of the solution. The stochastic convolution with respect to fractional Brownian motion is defined pathwise as a Riemann--Stieltjes integral, using the finite-variation properties of the fundamental solution. This representation allows us to compute the mean and auto-covariance function of the process. Under a delay-independent stability condition on the drift coefficients, we prove that the fundamental solution decays exponentially and derive the long-time behaviour of the stochastic delay equation. In particular, the solution converges in finite-dimensional distributions to a stationary Gaussian process, whose limiting mean and covariance are given explicitly. When the Hurst parameter satisfies \(H>1/2\), we show that the limiting covariance preserves the long-memory behaviour of the driving fractional Brownian motion and decays asymptotically as \(h^{2H-2}\). Finally, by extending the noise to a double-sided fractional Brownian motion, we construct a stationary solution of the equation and prove its uniqueness. The results provide a tractable framework for modelling stationary systems with both delay effects and fractional memory, with potential applications to stochastic volatility, energy modelling, and other time series exhibiting persistence. \end{abstract}

\newpage
\section{Introduction}

Continuous-time models driven by fractional Brownian motion provide a natural framework for stochastic systems whose fluctuations are not well described by the Markovian and semimartingale structure of classical Brownian motion. Since the seminal construction of fractional Brownian motion by \citet{Mandelbrot1968}, the Hurst parameter has become a parsimonious way to encode both regularity and temporal dependence. In financial applications this feature is particularly relevant for volatility modelling. The fractional stochastic volatility model of \citet{comte1998long} introduced long-range dependence through a fractional Ornstein--Uhlenbeck factor, while the empirical evidence of \citet{gatheral2022volatility} suggested that log-volatility is much rougher than standard Brownian motion, with a Hurst parameter well below one half. This observation led to a large rough-volatility literature, including option-pricing models such as \citet{Bayer2016} and empirical models that decouple short-run roughness from long-run persistence, as in \citet{bennedsen2022decoupling}. Fractional Ornstein--Uhlenbeck processes remain a central tractable building block in this literature; see, for instance, \citet{Cheridito2003} and the recent applications to variance and volatility derivatives in \citet{KIM2024102155}.

A second source of non-Markovian behaviour arises from delay. In many dynamical systems the present drift depends not only on the current state but also on past values of the state variable. This mechanism is especially relevant when feedback, information transmission, memory effects, or delayed adjustment play a role. Linear stochastic delay differential equations therefore extend the classical Langevin equation by allowing the drift to contain discrete lags. Foundational contributions in this direction include \citet{Kuchler1992}, who studied a delayed Langevin equation, and \citet{GUSHCHIN1999}, who analysed stationary solutions of delay equations driven by L\'evy noise. Statistical aspects of affine stochastic delay models were considered by \citet{Kulcher2013}. More recently, stochastic delay structures have appeared in financial and energy applications, including stochastic-delay volatility models \citep{GuineaJulia2024}, short-rate models with several delays \citep{Guinea2024closed}, and continuous-time stationary models for wind power production \citep{rohde2019continuous}.

The combination of fractional noise and delay is mathematically appealing because it brings together two distinct forms of memory. Fractional Brownian motion introduces dependence through the driving signal, whereas delay introduces dependence through the drift. Existing work on fractional stochastic differential equations and fractional stochastic delay equations includes \citet{rascanu2002differential}, \citet{Ferrante2006}, \citet{Rao2008}, \citet{Davis2020}, and \citet{ghani2025singular}. However, explicit distributional results for affine delay equations driven directly by fractional Brownian motion are still comparatively scarce, especially in the presence of several discrete delays and for a general Hurst parameter $H\in(0,1)$. The objective of this paper is to fill this gap for a linear multi-delay model that can be used as a Gaussian volatility factor or as a tractable component in more elaborate stochastic volatility specifications.

We study the fractional Ornstein Uhlenbeck with several discrete delay parameters. First in Section \ref{sect:model}, we obtain an analytical representation of the model based on the variation of constants formula. The main contribution of Section \ref{sect:limiting} is an explicit description of the  mean and autocovariance for arbitrary $H\in(0,1)$ and we also show that the existence of a limiting distribution, under some conditions on the parameters. We prove that the long-memory behavior of the fractional Brownian motion with $H>1/2$ is inherited by the limiting stationary process, in Section \ref{sect:LongMemory}. Thus the delays modify the multiplicative constant but do not change the power-law exponent generated by the fractional Brownian motion. This result clarifies the role of delay in a fractional volatility factor: delays affect the strength and shape of dependence at finite horizons, while the asymptotic memory exponent remains determined by $H$. Finally in Section \ref{sect:Stationary}, we construct a unique stationary solution of the model on the real line by using a two-sided fractional Brownian motion.

\section{Model}\label{sect:model}
Let $\left(\Omega,\mathcal{F},\mathds{P},\left(\mathcal{F}_t\right)_{t\geq 0}\right)$ be a filtered probability space. In this space, we study the solution of the following stochastic delay differential equation
\begin{equation}
dY_t =  \left(a + b Y_t + \sum_{j=1}^N c_j Y_{t-\tau_j}\right)dt + \sigma dW_t^H \label{eq:Y}
\end{equation} for all $t\ge0$, where $N\in\mathds{}{N}$, $a,b,c_1,c_2,\ldots,c_N\in\mathds{R}$, $\sigma>0$, $W^H=(W^H_t)_{t\geq0}$ is a fractional Brownian motion with Hurst exponent $H\in(0,1)$, and $\tau_N >\tau_{N-1}> \ldots> \tau_1>0$. The initial condition is
\begin{equation}
 Y_t = \phi(t) \text{ for }  t\in [-\tau_N,0], \label{eq:phi}
\end{equation}
where $\phi:[-\tau_N,0]\rightarrow\mathds{R}$ is a deterministic continuous function. 

An explicit solution for equation \eqref{eq:Y} will be obtained in Proposition \ref{prop:sol:SDE} below. To this end, we introduce the function $R:\mathds{R}\rightarrow\mathds{R}$, defined as 
\begin{equation} \label{eq:R}
 R(t) = 
 \begin{cases}
 \displaystyle \sum_{n=0}^{\lfloor\frac{t}{\tau_1}\rfloor}  \sum_{|\alpha|= n} \frac{c^\alpha}{\alpha !} \left(t- \langle \alpha,\tau \rangle\right)^n e^{b\left(t-\langle \alpha,\tau \rangle\right)} \mathcal{H}\left(t- \langle \alpha,\tau \rangle\right)   &\text{ when  }  t\ge0,\\
 0  &\text{ when  }t\in(-\infty,0),
 \end{cases} 
 \end{equation} where $\mathcal{H}=\mathds{1}_{[0,\infty)}$ is the Heaviside function and the inner summation in \eqref{eq:R} is expressed using multi-index notation with $\alpha=(\alpha_1,\alpha_2,\ldots,\alpha_N)$,
 \[ |\alpha| = \sum_{i=1}^N \alpha_i,\text{ }c^\alpha= \prod_{i=1}^N c_i^{\alpha_i},\text{ } \alpha!= \prod_{i=1}^N \alpha_i!, \text{ }  \langle \alpha,\tau \rangle = \sum_{i=1}^N \alpha_i \tau_i,  \]
 and $\alpha_i\geq 0$ for every $i=1,\ldots,N$. Notice that $R(0)=1$ and that $R$ does not depend on the parameter $a$. 
 
 Direct differentiation shows that $R$ satisfies the delay initial value problem
\begin{equation}
\left.\begin{aligned}
    R'(t) & = bR(t) + \sum_{j=1}^N c_j R(t-\tau_j)  && \text{ for almost everywhere } t>0, \\
    R(t) & = \mathds{1}_{\{0\}}(t) &&\text{when } t\in (-\infty,0].
    \end{aligned}\right\} \label{eq:Rdeter}
\end{equation}
More generally, $R$ is related to the solution of the deterministic delay initial value problem
 \begin{align} 
    x'(t) & = bx(t) + \sum_{j=1}^N c_j x(t-\tau_j) \text{ when } t>0 \label{eq:deter}\\
    x(t) & = g(t) \text{ when } t\in [-\tau_N,0], \nonumber
\end{align} where $g:[-\tau_N,0] \to \mathds{R}$ is a continuous deterministic function. It is given by the following result. 
 
\begin{proposition}\citep[Proposition 2.1]{Guinea2024closed}\label{prop:DeterDelay}
    The unique solution of equation \eqref{eq:deter} is
    \[ x(t) = x(0) R(t) + \sum_{j=1}^N c_j \int_{0}^{t\wedge \tau_j} R(t-s) g(s-\tau_j) ds \text{ for } t\geq 0,\] where $R$ is defined as in equation  \eqref{eq:R}.
\end{proposition}

The function $R$ is called the fundamental solution of \eqref{eq:deter}; see Section 1.5 of \citet{hale2006functional}. 

The final step before obtaining the solution to equation \eqref{eq:Y} is to study the stochastic integral of the function $R$ with respect to the fractional Brownian motion $W^H$.

\begin{remark}\label{remark:Integrlal1}
The function $R$ is locally of bounded variation since its derivatives is piecewise continuous. This means that we can define the stochastic integral of the function $R$ with respect to the fractional Brownian motion $W^H$ using integration by parts. By the finite variation property of $R$, the integral \[ \int_0^t W_s^H dR(t-s)\] exists in a Riemann-Stieltjes sense for every $t>0$. By applying Theorem 2.21 in \citet{wheeden1977measure}, we have that the stochastic integral
\[ \int_0^t  R(t-s)dW_s^H \text{ for } t\geq 0 \] is defined as a Riemann-Stieltjes integral, and it satisfies 
\begin{align}
    \int_0^t  R(t-s)dW_s^H  & =  W_t^H R(0)  -\int_0^t W_s^H dR(t-s)\nonumber\\
    &=  \int_0^t R'(t-u) W_u^H du + W_t^H \text{ for } t\geq  0. \label{eq:StochInt1} 
\end{align}
\end{remark} 

Finally, we have the following. 

\begin{proposition} \label{prop:sol:SDE}
The solution $(Y_t)_{t\ge0}$ of the stochastic differential equation \eqref{eq:Y} with initial condition \eqref{eq:phi} is given by
\begin{multline} 
 Y_t  = R(t)Y_0 + a\int_0^t R(t-s)ds + \sum_{j=1}^N c_j\int_{-\tau_j}^0 R(t-s-\tau_j)\phi(s)ds \\
 + \sigma \int_0^t  R(t-s)dW_s^H \label{eq:2.5}
\end{multline}
for all $t\ge0$, where $R$ is defined in \eqref{eq:R}. Furthermore, a change of variables allows us to express \eqref{eq:2.5} as
\begin{multline*} 
 Y_t  = R(t)Y_0 + a\int_0^t R(t-s)ds + \sum_{j=1}^N c_j\int_{0}^{t\wedge \tau_j} R(t-s)\phi(s-\tau_j)ds \\
 + \sigma \int_0^t  R(t-s)dW_s^H.
\end{multline*}
\end{proposition}

\begin{proof}
First, notice that the process in \eqref{eq:2.5} is well defined by Remark \ref{remark:Integrlal1}. Existence and uniqueness come from applying the method of steps; see  \citet[p.~158]{mao1997} for the method being applied to stochastic differential equations with standard Brownian motion. 

We will show that the process defined in \eqref{eq:2.5} satisfies equation \eqref{eq:Y}. It follows from \eqref{eq:StochInt1} that
\begin{equation}\label{eq:decomp}
    Y_t =  f(t) + \sigma W_t^H \text{ for all } t\geq 0,
\end{equation} where $f$ is defined as
\begin{align*}
    f(t) &= R(t)Y_0 + a\int_0^t R(t-s)ds + \sum_{j=1}^N c_j\int_{-\tau_j}^0 R(t-s-\tau_j)\phi(s)ds \nonumber \\
    & \qquad + \sigma \int_0^t R'(t-u) W_u^H du \nonumber \\
    &= R(t)Y_0 + a\int_0^t R(t-s)ds + \sum_{j=1}^N c_j\int_{-\tau_j}^0 R(t-s-\tau_j)\phi(s)ds \nonumber \\
    & \qquad + b\sigma \int_0^t R(t-u) W_u^H  du + \sigma \sum_{j=1}^N c_j\int_0^t R(t-u-\tau_j) W_u^H du \label{eq:f-def}
\end{align*}
as $R$ satisfies the initial value problem \eqref{eq:Rdeter}. 

Now define the  process $\widehat{W}$ as
\[\widehat{W}_t =W_t^H \mathds{1}_{\{t\geq 0 \}} \text{ for } t \in \mathds{R}.\]
Since $R(s) =0$ for $s<0$, a change of variables gives
\[\int_0^t R(t-u-\tau_j) W_u^H du = \int_{\tau_j}^{t+\tau_j} R(t-s) W_{s-\tau_j}^H ds = \int_{0}^t R(t-s) \widehat{W}^H_{s-\tau_j} ds \] for all $j=1,2,\ldots,N$.
After substitution into \eqref{eq:Rdeter} and applying the Leibniz rule, we arrive at
\begin{align}
     f'(t) & = R'(t)Y_0  + a\left[R(0) + \int_0^t R'(t-s) ds\right]\nonumber\\
    & \qquad + \sum_{j=1}^N c_j \left[\phi(t-\tau_j) \mathds{1}_{\{ t<\tau_j\}} +  \int_{0}^{t\wedge \tau_j} R'(t-s)\phi(s-\tau_j)ds\right]  \nonumber \\
    & \qquad  +  b\sigma \left[R(0)W_t^H  + \int_0^t R'(t-s) W_s^H  ds\right]  \nonumber \\
    & \qquad  + \sigma\sum_{j=1}^N c_j \left[R(0) \widehat{W}_{t-\tau_j}^H   + \int_0^t  R'(t-s) \widehat{W}_{s-\tau_j}^H ds\right] \label{eq:f1}
\end{align} for almost all $t>0$.
Equation \eqref{eq:StochInt1}, together with $R(0)=1$, gives
\begin{equation}
R(0)W_t^H + \int_0^t R'(t-s) W_s^H ds  = \int_0^t R(t-s)dW_s^H \label{eq:b1}.  
\end{equation}
Moreover, when $t>\tau_j$, and by a change of variables, we have for all $j=1,\ldots,N$ that 
\begin{align}
 R(0)  \widehat{W}_{t-\tau_j}^H
 + \int_0^t  R'(t-s)  \widehat{W}_{s-\tau_j}^H ds
 & = R(0)  {W}_{t-\tau_j}^H
 + \int_{\tau_j}^{t}  R'(t-s)  W_{s-\tau_j}^H ds\nonumber\\
  & = R(0)  {W}_{t-\tau_j}^H
 + \int_{0}^{t-\tau_j}  R'(t-\tau_j-s)  W_{s}^H ds\nonumber\\
 & =  \int_0^{(t-\tau_j)^+}  R(t-\tau_j-s) dW_{s}^H.\label{eq:c1}
\end{align} Equations \eqref{eq:b1}--\eqref{eq:c1} and the initial value problem \eqref{eq:Rdeter} allow us to rewrite \eqref{eq:f1} as
\begin{align}
     f'(t) & = \left[bR(t) + \sum_{j=1}^N c_j R(t-\tau_j)\right]Y_0 \nonumber \\
     &\qquad + a\left[1 + b \int_0^t R(t-s) ds + \sum_{j=1}^N c_j \int_0^t R(t-s-\tau_j)ds \right]\nonumber\\
    & \qquad + \sum_{j=1}^N c_j \phi(t-\tau_j) \mathds{1}_{\{ t<\tau_j\}}  + b\sum_{j=1}^N c_j \int_{0}^{t\wedge \tau_j} R(t-s)\phi(s-\tau_j)ds   \nonumber \\
   & \qquad + \sum_{j=1}^N \sum_{k=1}^N c_j c_k \int_{0}^{t\wedge \tau_j} R(t-s-\tau_k) \phi(s-\tau_j)ds  \nonumber \\
    & \qquad  +  b\sigma \int_0^t R(t-s)dW_s^H  + \sigma\sum_{j=1}^N c_j \int_0^{(t-\tau_j)^+}  R(t-\tau_j-s) dW_{s}^H. \label{eq:fPrima1}
\end{align}
Grouping the terms with coefficients $b$ and $c_j$, and swapping the indices $j$ and $k$ in the double summation, we obtain
\begin{align}
    f'(t) & =  a + bY_t \nonumber\\
    &+ \sum_{j=1}^N  c_j \left[\phi(t-\tau_j) \mathds{1}_{\{ t<\tau_j\}} + Y_0R(t-\tau_j) +  a \int_0^{(t-\tau_j)^+} R(t-\tau_j-s) ds \right.\nonumber\\
    & + \left. \sum_{k=1}^N c_k \int_{0}^{t\wedge \tau_k} R(t-s-\tau_j)\phi(s-\tau_k)ds +\sigma \int_0^{(t-\tau_j)^+} R(t-\tau_j-s) dW_{s}^H \right].
    \label{eq:fPrima2} 
\end{align} 

Notice that $Y_{t-\tau_j} = \phi(t-\tau_j)$ for $t\leq \tau_j$ and for $t> \tau_j$ we have
\begin{align*}
 Y_{t-\tau_j} & = Y_0R(t-\tau_j) +  \int_0^{t-\tau_j} a R(t-\tau_j-s) ds \nonumber\\
    & \qquad +  \sum_{k=1}^N c_k \int_{0}^{(t-\tau_j)\wedge \tau_k} R(t-s-\tau_j)\phi(s-\tau_k)ds \\
    & \qquad +\int_0^{t-\tau_j} \sigma R(t-\tau_j-s) dW_{s}^H,
\end{align*} where we used the fact that
\begin{align*}
     \int_{0}^{t\wedge \tau_k} R(t-s-\tau_j)\phi(s-\tau_k)ds &= \int_{-\tau_k}^{0} R(t-\tau_j-\tau_k-s)\phi(s)ds \\
     & = \int_0^{(t-\tau_j)\wedge \tau_k} R(t-\tau_j - u)\phi(u-\tau_k)du.
\end{align*}
This allows us to write equation \eqref{eq:fPrima2} as
\begin{equation}
    f'(t)  =  a + bY_t + \sum_{j=1}^N c_j Y_{t-\tau_j} \text{ for almost everywhere } t>0. \label{eq:fPrima3} 
\end{equation}
Equations \eqref{eq:fPrima3} and \eqref{eq:decomp} show that the process $Y$ in \eqref{eq:2.5} satisfies the stochastic delay differential equation \eqref{eq:Y}.
\end{proof} 

\section{Limiting distribution}\label{sect:limiting}

In this section, we are interested in studying the distributional properties of the process $Y_t$ given in Proposition \ref{prop:sol:SDE}. The properties of the fractional stochastic integral \citep[Chapter 6]{biagini2008stochastic} mean that $Y$ is a Gaussian process. Therefore, to obtain a complete understanding of $Y$, it is enough to study its first two moments. To this end, we have the following result.

\begin{proposition}\label{prop:Moments12}
    The expected value and auto-covariance of $Y_t$ for $t,h\geq 0$ is given by
    \begin{equation}
    E[Y_t] = R(t)Y_0 + a\int_0^t R(t-s)ds + \sum_{j=1}^N c_j\int_{-\tau_j}^0 R(t-s-\tau_j)\phi(s)ds\label{eq:Expected}
    \end{equation}
    and
    \begin{align}
& \operatorname{Cov}\left[ Y_t, Y_{t+h}\right] \nonumber\\
&= \sigma^2 \left( -\frac{1}{2}h^{2H} + \frac{1}{2}t^{2H} R(t+h) + \frac{1}{2}(t+h)^{2H} R(t) \vphantom{\int_0^{t+h}}\right.\nonumber\\
     &\qquad-\frac{1}{2} \int_{-h}^t |s|^{2H} R'(s+h) ds -\frac{1}{2} \int_{0}^t |s+h|^{2H} R'(s) ds\nonumber\\
     &\qquad + \frac{1}{2} R(t+h) \int_0^t s^{2H} R'(t-s) ds  + \frac{1}{2} R(t)\int_0^{t+h} s^{2H} R'(t+h-s) ds \nonumber\\
     &\qquad\left. - \frac{1}{2} \int_0^t \int_0^{t+h} |s-u|^{2H} R'(t-s) R'(t+h-u) du\, ds\right).\label{eq:AutoCov}
    \end{align}
\end{proposition}

When the Hurst exponent satisfies $H>1/2$, a more compact representation of the auto-covariance of  $Y$ exists than the one given in Proposition \ref{prop:Moments12}; see \eqref{eq-LongAutoCov} in Section \ref{sect:LongMemory} below.

\begin{proof}[Proof of Proposition \ref{prop:Moments12}.] 
    Equation \eqref{eq:Expected} follows directly from properties of the fractional stochastic integral \citep[Chapter 6]{biagini2008stochastic}. Equation \eqref{eq:AutoCov} requires more work. Applying \eqref{eq:Expected} and \eqref{eq:StochInt1}, and using the auto-covariance function of fractional Brownian motion \citep[Chapter 6]{biagini2008stochastic} gives
    \begin{align}
        &\operatorname{Cov}\left[ Y_t, Y_{t+h}\right] \nonumber\\
        & = \sigma^2 E\left[ \left(W_t^H + \int_0^t W_s^H R'(t-s) ds \right) \left(W_{t+h}^H + \int_0^{t+h} W_s^H R'(t+h-s) ds \right)  \right] \nonumber \\
        & = \sigma^2\left( \frac{1}{2}\left( (t+h)^{2H} + t^{2H} -h^{2H}  \right) + E\left[ W_t^H \int_0^{t+h} W_s^H R'(t+h-s) ds \right]   \right.\nonumber\\
        & \qquad + E\left[ W_{t+h}^H \int_0^{t} W_s^H R'(t-s) ds \right]\nonumber\\
        &\left. \qquad +  E\left[  \int_0^{t} W_s^H R'(t-s) ds \int_0^{t+h} W_u^H R'(t+h-u)du \right] \right). \label{eq:autoCov2}
    \end{align}
    
    We now consider each of the expectation terms in turn. Combining Fubini's Theorem with the properties of fractional Brownian motion, the first expectation becomes
\begin{align*}
    & E\left[ W_t^H \int_0^{t+h} W_s^H R'(t+h-s) ds \right] \nonumber \\
    & = \int_0^{t+h} \frac{1}{2}\left( t^{2H} + s^{2H} - |t-s|^{2H} \right)  R'(t+h-s) ds\nonumber\\
    & = \frac{1}{2} t^{2H} (R(t+h)-1) + \frac{1}{2} \int_0^{t+h} s^{2H} R'(t+h-s) ds - \frac{1}{2} \int_{-h}^t |y|^{2H} R'(y+h) dy.
\end{align*} Similarly, for the second expected value, we obtain
\begin{multline*}
    E\left[ W_{t+h}^H \int_0^{t} W_s^H R'(t-s) ds \right] \\
    = \frac{1}{2} (t+h)^{2H} (R(t)-1) + \frac{1}{2} \int_0^t s^{2H} R'(t-s)ds - \frac{1}{2} \int_0^t (y+h)^{2H} R'(y)dy.
\end{multline*}
Finally, the last expected value in \eqref{eq:autoCov2} can be expressed as
\begin{multline*}
    E\left[  \left(\int_0^{t} W_s^H R'(t-s) ds\right) \left(\int_0^{t+h} W_u^H R'(t+h-u)du\right) \right]\nonumber\\
    \begin{aligned}
    &= \int_0^t \int_0^{t+h} \frac{1}{2}\left(s^{2H} +u^{2H} - |s-u|^{2H}\right) R'(t-s) R'(t+h-u) du ds \nonumber\\
    & = \frac{1}{2} \left(R(t+h) -1\right)  \int_0^t s^{2H} R'(t-s) ds \nonumber\\
    & \qquad +  \frac{1}{2} \left(R(t) -1\right)  \int_0^{t+h} u^{2H} R'(t+h-u) du \nonumber\\
    & \qquad - \frac{1}{2} \int_0^t \int_0^{t+h} |s-u|^{2H} R'(t-s) R'(t+h-u) du ds.
    \end{aligned}
\end{multline*} Combining this with \eqref{eq:autoCov2} gives \eqref{eq:AutoCov}.
\end{proof}

We now turn our attention to the behavior of the process \(Y\) as \(t \to \infty\). The asymptotic behavior of \(Y\) depends on the stability properties of the delay differential equation \eqref{eq:deter} through the function $R$. The relevant notion of stability is known as asymptotic stability, which, intuitively, means that the trivial solution $x=0$ of \eqref{eq:deter} attracts other solutions as time goes to infinity. This means that perturbations in the initial condition $f$ (and influence of the initial history $\phi$ in our model) become insignificant as time increases. The following characterisation of asymptotic stability is due to \cite[Corollary~3.4]{HALE1985533}. The definition of asymptotic stability is standard \cite[cf.~e.g.][p.~38]{smith2011introduction}.

\begin{proposition}
 The following statements are equivalent for the delay differential equation \eqref{eq:deter}:
 
 \begin{enumerate}
   \item We have
    \begin{align}\label{eq:cond1}
    b &< 0, & b + \sum_{j=1}^N c_j &\neq 0, & |b| &\geq  \sum_{j=1}^N |c_j|.
\end{align}
    
    \item Equation \eqref{eq:deter} is asymptotically stable for every choice of the delay parameters $\tau_1,\ldots,\tau_N>0$, in other words, for every $\varepsilon > 0$, there exists $\delta > 0$
    such that
    \[
        \sup_{s \in [-\tau_{N},0]} |g(s)| < \delta
        \quad \Longrightarrow \quad
        |x(t)| < \varepsilon
        \qquad \text{for all } t \geq 0.
    \]
    and there exists $\delta_{0} > 0$ such that
    \[
       \sup_{s \in [-\tau_{N},0]} |g(s)|< \delta_{0}
        \quad \Longrightarrow \quad
        \lim_{t \to \infty} x(t) = 0.
    \]
 \end{enumerate}
\end{proposition}

If equation \eqref{eq:deter} is asymptotically stable for every choice of the delay parameters $\tau_1,\ldots,\tau_N>0$, then there exist constants $M,\lambda >0$  such that
\begin{equation}\label{eq:Rbound}
    |R(t) | \leq M e^{-\lambda t}  \text{ for  } t\geq 0
\end{equation}
\citet[Corollary~6.1]{hale2006functional}. We then have the following.

\begin{lemma}
For every choice of the delay parameters $\tau_1,\ldots,\tau_N>0$, if condition \eqref{eq:cond1} holds, then
\begin{align}
    \left|R'(t)\right| & \leq  C_R e^{-\lambda t} \text{ for almost everywhere } t\geq 0, \label{eq:Rprimebound}
\end{align} where $C_R = M  \left( |b| + \sum_{j=1}^N |c_j| e^{\lambda \tau_j}\right)$, and $\lambda>0$ and $M>0$ are the same constants as in \eqref{eq:Rbound}.

\end{lemma}

\begin{proof}  From the initial value problem \eqref{eq:Rdeter}, and  inequality \eqref{eq:Rbound}, we arrive at
    \begin{align*}
    \left|R'(t)\right| & \leq |b| |R(t)| + \sum_{j=1}^N |c_j| |R(t-\tau_j)| \\
    & \leq  |b| M  e^{-\lambda t} + \sum_{j=1}^N |c_j| M e^{-\lambda (t-\tau_j)} 
     = C_R e^{-\lambda t} 
\end{align*} for almost everywhere $t\ge 0$.
\end{proof}

We are finally ready to establish the asymptotic distribution of the process~$Y$. 

\begin{proposition}\label{prop:stationary_process}Assume that condition \eqref{eq:cond1} holds. Then the process $Y$ converges in
finite-dimensional distributions to a strictly stationary Gaussian process  $Y^\infty=(Y^\infty_s)_{s\ge0}$ with mean 
     \begin{equation}\label{eq:limit_mean}
       \mu =  a\int_0^\infty R(s) ds,
     \end{equation}and auto-covariance function 
     \begin{multline}
       \gamma(h) =  \sigma^2 \left(-\frac{1}{2}h^{2H}    -\frac{1}{2} \int_{-h}^\infty |y|^{2H} R'(y+h) dy -\frac{1}{2} \int_{0}^\infty |y+h|^{2H} R'(y) dy \right. \\
         \left. - \frac{1}{2} \int_0^\infty \int_{-h}^{\infty} |s-u|^{2H} R'(s) R'(u+h) du ds \right),\label{eq:limit_covar}
     \end{multline} for $h\geq 0$, and for every choice of delay parameters $\tau_1,\ldots,\tau_N >0$.
\end{proposition}
\begin{proof} To prove this result, we show that all terms in \eqref{eq:Expected}--\eqref{eq:AutoCov} either converge or vanish as $t\to \infty$. 

Equations \eqref{eq:Rbound} and \eqref{eq:Rprimebound} imply the following asymptotic relations:
\begin{align}
        \lim_{t\to \infty} |R(t)| &=0, \label{eq:Limit1}\\
        \lim_{t\to \infty}| R'(t)| & =0, \label{eq:Limit2}\\
        \lim_{t\to \infty} \left|t^{2H} R(t)\right| & = 0, \label{eq:Limit3}\\
       \lim_{t\to \infty} \left|R(t)\int_0^t s^{2H} R'(t-s) ds\right| & = 0.\label{eq:Limit4}
       \end{align}

The first and third terms in \eqref{eq:Expected} vanish as $t\to\infty$ due to \eqref{eq:Limit1} and the exponential bound \eqref{eq:Rbound}. Furthermore,
\[
\int_0^t R(t-s)\,ds = \int_0^t R(u)\,du \to \int_0^\infty R(u)\,du \text{ as } t\to \infty.
\]
Hence \eqref{eq:limit_mean} follows.

We now turn to \eqref{eq:AutoCov}. For the first integral, combining \eqref{eq:Rbound} with the definition of the Gamma function gives
\begin{align*}
    \int_{-h}^\infty |y|^{2H} \left| R'(y+h)\right| dy 
    &\leq C_{R} e^{-\lambda h }  \int_{-h}^\infty |y|^{2H} e^{-\lambda y } dy\nonumber\\
     & =  C_{R} e^{-\lambda h } \left(  \int_{-h}^0 (-y)^{2H} e^{-\lambda y } dy + \int_{0}^\infty y^{2H} e^{-\lambda y } dy \right) \nonumber\\
     & = C_{R} e^{-\lambda h } \left(  \int_{-h}^0 (-y)^{2H} e^{-\lambda y } dy + \frac{\Gamma(2H +1)}{\lambda^{2H+1}}\right) <\infty.
\end{align*} 
Dominated convergence then gives
\begin{align}
    \lim_{t \to \infty} \int_{-h}^t |y|^{2H} R'(y+h) dy   = \int_{-h}^\infty |y|^{2H} R'(y+h) dy\label{eq:Convergence_Integral_1}
\end{align}
 Similarly, we obtain
\begin{align*}
    \int_{0}^\infty |y+h|^{2H} \left|R'(y) \right| dy 
      \leq C_R e^{\lambda h}\left(- \int_{0}^{h} v^{2H} e^{-\lambda v} dv +  \frac{\Gamma(2H +1)}{\lambda^{2H+1}}  \right) < \infty,
\end{align*} which leads to
\begin{equation}
    \lim_{t \to \infty}  \int_{0}^t |y+h|^{2H} R'(y) dy =   \int_{0}^\infty |y+h|^{2H} R'(y)  dy.  \label{eq:Convergence_Integral_2}
\end{equation}
For the double integral \eqref{eq:AutoCov}, applying \eqref{eq:Rbound} and making the successive changes of variables $\sigma = t-s$ and $v= t-u$ and $r=v+\sigma$, $p=v-\sigma$, we obtain
\begin{multline*}
    \lim_{t \to \infty} \left|  \int_0^t \int_0^{t+h} |s-u|^{2H} R'(t-s) R'(t+h-u) du ds \right|\nonumber\\
    \begin{aligned}
    & \leq  C_R^2 e^{-\lambda h} \lim_{t \to \infty}   \int_0^t \int_0^{t+h} |s-u|^{2H} e^{-\lambda(t-s)}  e^{-\lambda(t-u)}du ds  \nonumber \\
    & = C_R^2 e^{-\lambda h} \lim_{t \to \infty}   \int_0^t \int_{-h}^{t} |\sigma-v|^{2H} e^{-\lambda(\sigma+v)}  d\sigma dv   \nonumber\\
    & = C_R^2 e^{-\lambda h}  \int_0^{\infty} \int_{-h}^{\infty} |\sigma-v|^{2H} e^{-\lambda(\sigma+v)}  d\sigma dv \\
    & = \frac{1}{2}C_R^2 e^{-\lambda h}   \int_{-\infty}^{\infty} |p|^{2H}  \int_{\max\{p,-2h-p \}}^{\infty} e^{-\lambda r}  dr dp\nonumber\\
     & = \frac{1}{2\lambda}C_R^2 e^{-\lambda h}   \int_{-\infty}^{\infty} |p|^{2H}  e^{-\lambda \max\{p,-2h-p \}} dp\nonumber\\
      & = \frac{1}{2\lambda}C_R^2 e^{-\lambda h}\left(   \int_{-\infty}^{-h} |p|^{2H}  e^{\lambda (2h+p)} dp + \int_{-h}^{\infty} |p|^{2H}  e^{-\lambda p} dp\right) \nonumber\\
      & = \frac{1}{2\lambda}C_R^2 e^{-\lambda h}\left(  e^{2\lambda h} \int_{h}^{\infty} y^{2H}  e^{-\lambda y} dy + \int_{-h}^{\infty} |p|^{2H}  e^{-\lambda p} dp\right) \\
      & = \frac{1}{2\lambda}C_R^2 e^{-\lambda h}\left( 2\int_0^h y^{2H}\cosh\left(\lambda y\right) dy + \left(1+e^{2\lambda h} \right) \int_{h}^{\infty} y^{2H}  e^{-\lambda y} dy \right).
      \end{aligned}
\end{multline*}
It then follows from the definition of the upper incomplete Gamma function $\Gamma(\cdot,\cdot)$ that
\begin{multline}
    \lim_{t \to \infty} \left|  \int_0^t \int_0^{t+h} |s-u|^{2H} R'(t-s) R'(t+h-u) du ds \right|\nonumber\\
    \leq   \frac{1}{2\lambda}C_R^2 e^{-\lambda h}\left( \int_0^h 2y^{2H}\cosh\left(\lambda y\right) dy + \left(1+e^{2\lambda h} \right) \frac{\Gamma(2H+1,\lambda h)}{\lambda^{2H +1}} \right)< \infty, \label{eq:Integral_3}
\end{multline} where $\Gamma(\cdot,\cdot)$ is the upper incomplete Gamma function. Finally, a change of variables and dominated convergence gives that
\begin{multline*}
 \lim_{t \to \infty} \int_0^t \int_0^{t+h} |s-u|^{2H} R'(t-s) R'(t+h-u) du ds \\
 \begin{aligned}
  &= \lim_{t \to \infty} \int_0^t \int_{-h}^{t} |\sigma - v|^{2H} R'(\sigma) R'(v+h) dv d\sigma \\
  &= \int_0^\infty \int_{-h}^{\infty} |\sigma - v|^{2H} R'(\sigma) R'(v+h) dv d\sigma.
 \end{aligned}
\end{multline*}
Combining this with \eqref{eq:Limit1}--\eqref{eq:Limit4}, \eqref{eq:Convergence_Integral_1} and \eqref{eq:Convergence_Integral_2} shows that every term in \eqref{eq:AutoCov} either converges or vanishes as $t\to \infty$. This proves the result in \eqref{eq:limit_covar}.

Finally, since \(Y\) is a Gaussian process, the vector
$
(Y_{t+s_1}, \ldots, Y_{t+s_m})
$
is Gaussian for every \(m\in\mathds{N}\) and all
\(0 \leq s_1 \leq \cdots \leq s_m\). We have proved that,
\begin{align*}
\lim_{t\rightarrow\infty} E[Y_{t+s_i}] &= \mu,
&
\lim_{t\rightarrow\infty} \operatorname{Cov}(Y_{t+s_i}, Y_{t+s_j}) &= \gamma(s_j - s_i). 
\end{align*}
Therefore, for every fixed $m$ and $0\le s_1\le\cdots\le s_m$, the vector
$
(Y_{t+s_1},\dots,Y_{t+s_m})
$
converges in distribution to a Gaussian vector with mean vector
$(\mu,\dots,\mu)$ and covariance matrix $\big(\gamma(s_j-s_i)\big)_{i,j=1}^m$.
Since the limiting covariance depends only on the lag, the limiting Gaussian
process is stationary. Hence $(Y_{t+s})_{s\ge0}$ converges in finite-dimensional
distributions to a strictly stationary Gaussian process.
\end{proof}

We can derive an explicit formula for the expected value $\mu$.

\begin{corollary}\label{cor:stat_mean}
    Assume that condition \eqref{eq:cond1} holds. Then
\begin{equation}\label{eq:limit_mean2}
       \lim_{t\to \infty} E[Y_t] =  a\int_0^\infty R(s) ds = -\frac{a}{b+\sum_{j=1}^N c_j} 
    \end{equation} for every choice of delay parameters $\tau_1,\ldots,\tau_N >0$.
    
    Furthermore, in the case $a=1$ we obtain
\begin{equation}\label{eq:InfIntegral}
\int_0^\infty R(s)\,ds = -\frac{1}{b+\sum_{j=1}^N c_j}.
\end{equation}
\end{corollary}
\begin{proof}
    The transformation
    \[ X_t = Y_t +\frac{a}{b+\sum_{j=1}^N c_j} \text{ for } t\in[-\tau_N,\infty),\] leads to the stochastic delay differential equation with initial condition
    \begin{align}
        dX_t & = \left(bX_t + \sum_{j=1}^N c_j X_{t-\tau_j} \right) dt + \sigma d W_t^H \text{ for } t>0\label{eq:X1}\\
        X_t & = \phi(t) + \frac{a}{b+\sum_{j=1}^N c_j} \text{ for } t\in [-\tau_N,0] \label{eq:X2}.
    \end{align} Observe that the process $X$ satisfies \eqref{eq:Y} with $a=0$. Applying Proposition \ref{prop:stationary_process} to the initial value problem \eqref{eq:X1}--\eqref{eq:X2}, we get  \[ \lim_{t\to \infty} E[X_t] =0,\] and the result in \eqref{eq:limit_mean2} follows.
\end{proof}
To conclude this section, we give the formula for the variance, which follows from Propositions \ref{prop:Moments12}--\ref{prop:stationary_process}.
\begin{remark}From equations \eqref{eq:AutoCov} and \eqref{eq:limit_covar}, we can compute the variance of $Y$ as
\begin{multline}
\operatorname{Var}[Y_t] =  \sigma^2 \left( t^{2H}R(t) - \int_0^t y^{2H} R'(y) dy +  R(t) \int_0^t s^{2H} R'(t-s) ds\right.\\
        \left. -  \int_0^t\int_0^t \frac{1}{2} |s-u|^{2H} R'(t-s) R'(t-u) ds du\right), \label{eq:var1}
\end{multline} and its asymptotic variance as
\begin{equation}
  \gamma(0) =  \sigma^2 \left( - \int_0^\infty y^{2H} R'(y) dy-  \int_0^{\infty}\int_0^{\infty} \frac{1}{2} |s-u|^{2H} R'(s) R'(u) ds du \right).\label{eq:limit_var}
\end{equation} 
    
\end{remark}

\section{Long memory}\label{sect:LongMemory}

In the long memory case, when $H>1/2$, the auto-covariance of process $Y$ can be written more compactly than \eqref{eq:AutoCov}, due to the isometry property of the fractional stochastic integral \citep[Chapter 6]{biagini2008stochastic}. In particular, when \(H>1/2\) we have
\begin{align} \label{eq-LongAutoCov}
 \operatorname{Cov}\left[Y_t,Y_{t+h}\right] = \sigma^2 H(2H-1) \int_0^t \int_0^{t+h} R(u)R(v) \left|h+u-v \right|^{2H-2} du\,dv.
\end{align}
In this section we will specialize to the case where \(H>1/2\), and assume that~\eqref{eq:cond1} holds. Our aim is to verify that, under these assumptions, the process \(Y\) preserves the long-memory property of the fractional Brownian motion as time tends to infinity. To this end, we use the fact that
\begin{align}
    \gamma(h) &= \lim_{t\to\infty} \operatorname{Cov}\left(Y_t,Y_{t+h}\right)\nonumber\\
    &= \sigma^2 H(2H-1)\int_0^\infty \int_0^\infty R(u)R(v)\left|h+u-v\right|^{2H-2}\,du\,dv.
    \label{eq:AutcovarianceLong}
\end{align}
To show that the long-memory property is preserved when \(H>1/2\), we need to show that
\begin{equation} \label{eq:long-memory}
\gamma(h)\sim Ch^{2H-2}\qquad \text{as } h\to\infty.
\end{equation}
for some constant $C>0$ \citep[p.~42]{palma2007longmemory}. To this end, we have the following.

\begin{proposition}\label{prop:longAuto}
The auto-covariance function \(\gamma(h)\) defined in \eqref{eq:AutcovarianceLong} satisfies~\eqref{eq:long-memory}, where
\begin{equation}
    C=\frac{\sigma^2 H(2H-1)}{\left(b+\sum_{j=1}^N c_j\right)^2}.
\end{equation}
\end{proposition}
\begin{proof} 
    Let us write $\alpha = 2H - 2 \in (-1,0)$, and using \eqref{eq:AutcovarianceLong} define:
    \begin{align}\label{eq:int_def}
       I(h) := \frac{\gamma(h)}{\sigma^2 H(2H-1) h^\alpha} =  \int_0^\infty  \int_0^\infty  R(u) R(v) \left| 1+ \frac{u-v}{h} \right|^\alpha du\,dv.
    \end{align} To obtain the desired result, it is sufficient to show that
    \[ \lim_{h\to \infty}  I(h) =   \left(\int_0^\infty R(u) du \right)^2 = \frac{1}{\left(b+\sum_{j=1}^N c_j\right)^2}. \]
    Observe immediately that
    $\left| 1+ \frac{u-v}{h}\right|^\alpha \rightarrow 1$ as $h\rightarrow\infty$ because $\alpha \in (-1,0)$. However, there is a singularity in \eqref{eq:int_def} when $1+\frac{u-v}{h} \approx 0$, so we need to isolate that region when dealing with the integral. Let us define
   \begin{align*}
       E_h & := \left\{(u,v) \in [0,\infty)^2: \left| 1+\frac{u-v}{h}\right|<\frac{1}{2} \right\},\\
        E_h^c & := \left\{(u,v) \in [0,\infty)^2: \left| 1+\frac{u-v}{h}\right|\geq \frac{1}{2} \right\}
   \end{align*}
   and decompose the integral of interest as
   \[ I(h) = I_{E_h}(h) + I_{E_h^c}(h),\]where\begin{align*}
       I_{E_h}(h) & =  \iint_{E_h} R(u) R(v) \left| 1+ \frac{u-v}{h} \right|^\alpha du\,dv,\\
        I_{E_h^c}(h) & = \iint_{E_h^c} R(u) R(v) \left| 1+ \frac{u-v}{h} \right|^\alpha du\,dv.
   \end{align*}
   
   Let us first compute the integral $I_{E_h^c}(h)$. Since $\alpha<0$, we have \[\left| 1+\frac{u-v}{h}\right|^\alpha \leq \left(\frac{1}{2}\right)^\alpha\] on $E_h^c$. Moreover, a $\mathds{1}_{E^c_h}(u,v) \to 1$ as $h\to \infty$, by application of the Dominated Convergence Theorem, we arrive at 
\begin{align*}
     I_{E_h^c}(h) & = \int_0^\infty  \int_0^\infty \mathds{1}_{E^c_h}(u,v)  R(u) R(v) \left| 1+ \frac{u-v}{h} \right|^\alpha du\,dv\\
     & \to \left( \int_0^\infty R(s) ds\right)^2 \text{ as } h\to \infty.
\end{align*}

     We now show that $I_{E_h}(h)\to 0$ as $h\to \infty$. Note that
     \begin{equation}\label{eq:setEh}
         E_{h} =\{ (u,v) \in [0,\infty)^2: v\in [u+h/2,u+3h/2]\}.
     \end{equation} From \eqref{eq:setEh} and the bound in \eqref{eq:Rbound}, one obtains
   \begin{align}
       \left| I_{E_h}(h)\right| & \leq \int_0^\infty \int_{u + h/2}^{u + 3h/2} \left|R(u) R(v) \right|  \left| 1+\frac{u-v}{h}\right|^\alpha dv du \nonumber\\
       & \leq M^2 \int_0^\infty \int_{u + h/2}^{u + 3h/2} e^{-\lambda (u+v)} h^{-\alpha}  \left| h+u-v\right|^\alpha dv du.\label{eq:Eh1}
   \end{align}  For a fixed $u$, the change of variable $z=h+u-v$ gives us
   \begin{align}
       \int_{u + h/2}^{u + 3h/2} e^{-\lambda v}   \left| h+u-v\right|^\alpha dv & = e^{-\lambda (u+h)} \int_{-h/2}^{h/2} e^{\lambda z} |z|^\alpha dz \nonumber\\
       & \leq 2 e^{-\lambda (u+h/2)} \int_0^{h/2} z^\alpha dz \nonumber\\
       & = \frac{h^{\alpha+1}}{(\alpha +1)2^\alpha} e^{-\lambda (u+h/2)}. \label{eq:Eh2}
   \end{align} Combining \eqref{eq:Eh1} and \eqref{eq:Eh2}, we arrive at
    \begin{align}
       \left| I_{E_h}(h)\right| & \leq  \frac{M^2 h}{(\alpha +1)2^\alpha} e^{-\lambda h /2} h \int_0^{\infty} e^{-2\lambda u} du\nonumber\\
       &= \frac{M^2}{(\alpha +1)2^{\alpha+1} \lambda} e^{-\lambda h /2} h \to 0 \text{ as } h\to\infty. \nonumber 
   \end{align} 
   
   In summary, we showed that \[I(h) \to \left(\int_0^\infty R(u) du \right)^2 \text{ as } h\to \infty.\] Hence by the definition of $I(h)$ in \eqref{eq:int_def}, we have that
  \[\gamma(h)\sim \sigma^2 H(2H-1) h^{2H-2} \left(\int_0^\infty R(u) du \right)^2 \text{ as } h\to\infty.\] Finally, equation \eqref{eq:InfIntegral} yields the desired result.
\end{proof}

\section{Stationary solution}\label{sect:Stationary}

The aim of this section is to estable that there exists a unique stationary solution for the stochastic differential equation \eqref{eq:Y}. We continue with our assumption that the associated delay differential equation \eqref{eq:deter} is asymptotically stable for every choice of the delay parameters $\tau_1,\ldots,\tau_N>0$. in other words, that condition \eqref{eq:cond1} is satisfied. 

In this section, equation \eqref{eq:Y} is considered on the whole real line. To this end, define the process $W^H = (W_t^H)_{t\in \mathds{R}}$ to be a double-sided fractional Brownian motion \citep[p.~60]{mishura2017theory}. A process $Y=(Y_t)_{t\in\mathds{R}}$ is then called a solution if, for every $s<t \in \mathds{R}$,
\[
Y_t-Y_s
=
\int_s^t
\left(
a+bY_u+\sum_{j=1}^N c_j Y_{u-\tau_j}
\right)\,du
+
\sigma\left(W_t^H-W_s^H\right).
\]
We will show that the process $Y^{\stat} = (Y_t^{\stat})_{t\in \mathds{R}}$  which is defined as
\begin{equation} \label{eq:StatSol}
    Y^{\stat}_t  = \mu  + \sigma \int_{-\infty}^t R(t-s) dW_s^H \text{ with } \mu = -\frac{a}{b+\sum_{j=1}^N c_j},
\end{equation} is the unique stationary solution of the stochastic differential equation \eqref{eq:Y}.

First, we note that the stochastic integral that appears in \eqref{eq:StatSol} is well defined as a Riemann-Stieltjes integral, as shown in the next remark.

\begin{remark}
   Proceeding as in the proof of Proposition A.1 in  \citet{Cheridito2003}, we have that $W^H$ is almost surely continuous and that
\begin{equation}\label{eq:limit1}
\lim_{|s|\to \infty} \frac{W_s^H}{|s|^\eta}  = 0 \text{ almost surely  for all } \eta>H.
\end{equation}
By the bounds in equations \eqref{eq:Rbound}--\eqref{eq:Rprimebound}, we have that the Riemann integral
\begin{align}
\int_{-\infty}^t W_s^H R'(t-s) ds \label{eq:RiemannIntegral}
\end{align} is well defined for $t\in \mathds{R}$. To see this, notice that by equation \eqref{eq:limit1}, we can find a random $L=L(\omega)>0$ that satisfies,
\begin{align}
\left|W_s^H\right| \leq |s|^\gamma \text{ for } |s|\geq L.
\end{align} The bound in \eqref{eq:Rprimebound} then gives
\begin{align*}
\int_{-\infty}^t \left|W_s^H R'(t-s) \right| ds & \leq C_R \int_{-\infty}^t \left|W_s^H  \right| e^{-\lambda (t-s)} ds\\
& \leq C_R \int_{-L}^t |W_s^H| e^{-\lambda (t-s)} ds   + C_R \int_{-\infty}^{-L} \left|s\right|^{\gamma} e^{-\lambda (t-s)} ds \\&<\infty.
\end{align*} Hence the Riemann integral \eqref{eq:RiemannIntegral} is well defined. 

Finally, from Theorem 2.21 of \citet{wheeden1977measure} we obtain that the stochastic integral in \eqref{eq:StatSol} is well defined as a Riemann-Stieltjes integral, and it satisfies
\begin{align*}
\int_{-\infty}^t R(t-s) dW_s^H  &= \lim_{A \to -\infty}   \int_{A}^t R(t-s) dW_s^H\\
& = W_t^H - \lim_{A\to- \infty}W_A^H R(t-A) - \lim_{A\to -\infty} \int_{A}^t W_s^H dR(t-s)
\end{align*} for all $t$, and where we recall that $R(0)=1$. From  the exponential bound  in \eqref{eq:Rbound} and the result in \eqref{eq:limit1}, we obtain that
\[  \lim_{A\to- \infty}W_A^H R(t-A) =0 \text{ almost surely}.\]
Finally, since the integral \eqref{eq:RiemannIntegral} is well defined, we have that
\begin{align}\label{eq:FractionalIntegral}
\int_{-\infty}^t R(t-s) dW_s^H & = W_t^H + \int_{-\infty}^t W_s^H R'(t-s) ds.
\end{align}
  
\end{remark}

To prove that  $Y^{\stat}$ is the unique stationary solution of the differential equation \eqref{eq:Y}, we split the proof into two parts. In the first part, we prove that the process $Y^{\stat}$ is stationary and satisfies equation \eqref{eq:Y}. The second part is devoted to uniqueness.

\subsection{Stationarity and solution property}
Before showing that the process $Y^{\stat}$ defined in \eqref{eq:StatSol} satisfies the differential equation \eqref{eq:Y}, let us show that its distribution is strictly stationary. To that end, we compute the auto-covariance function of $Y^{\stat}$ using the $M$-operator; see Appendix \ref{sect:laplace}.  

\begin{proposition}\label{prop:stat}
    The process $Y^{\stat}$ defined in \eqref{eq:StatSol} is a strictly stationary process with mean and auto-covariance
    \begin{align}
        E\left[Y_t^{\stat}\right] & = \mu, \label{eq:EspStat} \\
        \gamma(h) := \operatorname{Cov}[Y_t^{\stat},Y_{t+h}^{\stat}] & = \frac{\sigma^2 \kappa_H}{2\pi} \int_{\mathds{R}} e^{i \lambda h} \left| \widehat{R}(\lambda) \right|^2 \left|\lambda \right|^{1-2 H} d\lambda,\label{eq:AutoCov_stat}
    \end{align} where
    \begin{align}
        \widehat{R} (w) &= \frac{1}{i w - b - \sum_{j=1}^N c_j e^{-iw \tau_j}},\label{eq:hatR}\\
        \kappa_H & = \sin(\pi H) \Gamma(2H+1).\nonumber
    \end{align}
\end{proposition}
\begin{proof}
    The expected value in \eqref{eq:EspStat} comes from the fact that the stochastic integral in \eqref{eq:StatSol} has expected value zero. To compute the result in \eqref{eq:AutoCov_stat}, we use the $M$-operator, and we proceed as in the proof of Proposition 2.1 in \citet{KIM2024102155}. To that end, we express the stochastic integral  in \eqref{eq:StatSol} as
    \[  \int_{-\infty}^t R(t-s) dW_s^H  =  \int_{-\infty}^\infty f_t(s) dW_s^H  \text{ for } t\in\mathds{R},\]
    where \[f_t(s)= R(t-s) \mathds{1}_{\{(t-s)\geq 0\}}.\] From Proposition \ref{prop:SpaceL2h} in Appendix \ref{sect:laplace} we have $R\in L^2_H(\mathds{R})$, and hence $f_t \in L^2_{H}(\mathds{R})$ for every $t\in \mathds{R}$. This implies that the $M$-operator acting on the function $f_t$ is well defined.

Using \eqref{eq:FracIntMop} in Appendix \ref{sect:laplace} and the Parseval–Plancherel identity, we can express the auto-covariance of $Y^{\stat}$ as
    \begin{align*}    \operatorname{Cov}\left[Y_t^{\stat},Y_{t+h}^{\stat}\right]  
    & = E\left[\sigma\left(\int_{-\infty}^\infty f_t(s) dW_s^H \right) \sigma\left(\int_{-\infty}^\infty f_{t+h}(s) dW_s^H \right)  \right]\\
    & = \sigma^2 \int_{-\infty}^\infty Mf_t(\lambda) \overline{Mf_{t+h}(\lambda)} d\lambda\\
    & = \frac{\sigma^2 }{2\pi} \int_{-\infty}^\infty \widehat{Mf_t}(\lambda) \overline{\widehat{Mf_{t+h}}(\lambda)} d\lambda \\
      & = \frac{\sigma^2 \kappa_H}{2\pi} \int_{-\infty}^\infty |\lambda|^{1-2H}\widehat{f_t}(\lambda) \overline{\widehat{f_{t+h}}(\lambda)} d\lambda,
    \end{align*} where
    \[ \widehat{f_t}(\lambda) = \int_{\mathds{R}} e^{-i \lambda u} f_t(u) du. \]
By  a change of variables, we arrive at
\begin{align*}
    \widehat{f_t}(\lambda) &= \int_{\mathds{R}} e^{-i \lambda u} f_t(u) du\\
    & = e^{-i \lambda t}\int_0^{\infty} e^{i \lambda
    v}  R(v) dv\\
     & = e^{-i \lambda t}\int_{\mathds{R}} e^{i \lambda
    v}  R(v) dv\\
    & =  e^{-i \lambda t} \widehat{R}(-\lambda).
\end{align*} Similarly, we have that
\begin{align*}
    \widehat{f_{t+h}}(\lambda) 
    & =  e^{-i \lambda (t+h)} \widehat{R}(-\lambda),
\end{align*} Finally, 
\begin{align*}
    \widehat{f_t}(\lambda) \overline{\widehat{f_{t+h}}(\lambda)} 
    & =  e^{-i \lambda t} \widehat{R}(-\lambda) e^{i \lambda (t+h)} \overline{\widehat{R}(-\lambda)} \\
    & = e^{i \lambda h} \left|\widehat{R}(-\lambda) \right|^2 = e^{i \lambda h} \left|\widehat{R}(\lambda) \right|^2,
\end{align*} where the last equality comes from the fact that $R$ is real-valued. Using the Laplace transform of $R$ in Appendix \ref{sect:laplace}, we arrive at the  result in \eqref{eq:AutoCov_stat}. So we have just proved that the process $Y^{\stat}$ is weakly stationary. Since $Y^{\stat}$ is a Gaussian process,  weak stationarity implies strict stationarity \citep[p. 123]{CramerLeadbetter1967}. 
\end{proof}

\begin{remark}\label{remarkdensity}
In Proposition \ref{prop:stat} we computed the power spectral density of the stationary process $Y^{\stat}$. It is directly related to the auto-covariance function $Y^{\stat}$ via the inverse Fourier transform, i.e.,
\[\gamma(h)  = \int_{\mathds{R}} e^{i u h} f_Y(u) du \text{ for } h\in \mathds{R}; \] where $f_Y$ is the power spectral density of the process $Y_{stat}$;  see Section 7.4 by \citet{CramerLeadbetter1967}. From the result obtained in Proposition \ref{prop:stat}, we have that
\[ f_Y(u) = \frac{\sigma^2 \kappa_H}{2\pi}  \left| \widehat{R}(u) \right|^2 \left|u\right|^{1-2 H}.  \] 
\end{remark}

The following result shows that the stationary process $Y^{\stat}$ maintains the long-memory behavior of the fractional Brownian motion when $H>1/2$.

\begin{corollary}
    Assume that condition \eqref{eq:cond1} holds and that $H>1/2$. Then the process $Y^{\stat}$ defined in \eqref{eq:StatSol} has long memory. More precisely,
    \[ \gamma(h) \sim C h^{2H-2} \text{ as } h\to \infty,\] where
     \[ C=\frac{\sigma^2 H(2H-1)}{\left(b+\sum_{j=1}^N c_j\right)^2}.\]
\end{corollary}

\begin{proof}
Since $H>1/2$, properties of the fractional Wiener integral give us
\begin{align}
\gamma(h) & = \operatorname{Cov}\left[Y_t^{\stat},Y_{t+h}^{\stat}\right] \nonumber\\
& = \sigma^2 E\left[\int_{-\infty}^t R(t-s) dW_s^H \int_{-\infty}^{t+h} R(t+h - r) dW_r^H \right] \nonumber\\ 
& = \sigma^2 H(2H-1) \int_{-\infty}^t \int_{-\infty}^{t+h} R(t-s) R(t+h-r) |s-r|^{2H-2} dr ds. \label{eq:integralLongMemory} 
\end{align} The change of variables $u=t-s$ and $v=t+h-r$ allows us to rewrite the auto-covariance in \eqref{eq:integralLongMemory} as
\begin{align*}
    \gamma(h) = \sigma^2 H(2H-1) \int_{0}^\infty \int_{0}^{\infty} R(u) R(v) |h+u-v|^{2H-2} du\,dv.
\end{align*} Hence the auto-covariance of $Y^{\stat}$ can be written as in equation \eqref{eq:AutcovarianceLong}. Hence, Proposition \ref{prop:longAuto} gives us the desired result. 
\end{proof}

Define the centered  process $X^{\stat} = \left(X^{\stat}_t\right)_{t\in \mathds{R}}$ as
\[ X_t^{\stat} = Y_t^{\stat} - \mu =  \sigma \int_{-\infty}^t R(t-s) dW_s^H \text{ for } t\in \mathds{R}. \] We will show that the process $X^{\stat}$ satisfies the differential equation \eqref{eq:Y} with $a=0$.

\begin{proposition}\label{Prop:Xstat}
    The process $X^{\stat}$ satisfies the stochastic differential equation \eqref{eq:Y} with $a=0$.
\end{proposition}
\begin{proof} For all $u<t \in \mathds{R}$ we have 
     \begin{align}
         X_t^{\stat} - X_u^{\stat} 
         & =  \sigma \int_{-\infty}^u \left( R(t-s) - R(u-s)\right) dW_s^H + \sigma \int_u^t R(t-s) dW_s^H.\label{eq:DiffXstat}
     \end{align} By the fundamental theorem of calculus, we obtain
     \begin{align*}
          R(t-s) - R(u-s)= \int_u^t \frac{\partial }{\partial r} R(r-s) dr = \int_u^t R'(r-s) dr,
     \end{align*} and so we can express \eqref{eq:DiffXstat} as
          \begin{align}
         X_t^{\stat} - X_u^{\stat} 
         & =  \sigma \int_{-\infty}^u \left( \int_u^t R'(r-s) dr\right) dW_s^H + \sigma \int_u^t R(t-s) dW_s^H \nonumber\\
         & = \sigma \int_u^t \left(\int_{-\infty}^u R'(r-s) dW_s^H\right) dr + \sigma \int_u^t R(t-s) dW_s^H, \label{eq:DiffXstat2}
     \end{align} where, in the last equality, we apply stochastic Fubini`s Theorem. The fundamental theorem of calculus and stochastic Fubini`s Theorem, together with $R(0)=1$,  allow us to write the second stochastic integral in \eqref{eq:DiffXstat2} as
     \begin{align}
         \int_u^t R(t-s) dW_s^H & =\int_u^t\left( R(t-s) - R(s-s) +1\right) dW_s^H \nonumber \\
         & = \int_u^t \left(\int_s^t R'(r-s) dr\right) dW_s^H + W_t^H -W_u^H \nonumber \\
         & = \int_u^t \left(\int_u^r R'(r-s) dW_s^H\right) dr + W_t^H -W_u^H. \label{eq:DiffXstat3}
     \end{align} Equations \eqref{eq:DiffXstat}--\eqref{eq:DiffXstat3} then give
        \begin{align}
         X_t^{\stat} - X_u^{\stat} 
         & =  \sigma \int_u^t \int_{-\infty}^r R'(r-s) dW_s^H dr + \sigma \left( W_t^H - W_u^H\right).\label{eq:DiffXstat4}
     \end{align} Since $R$ satisfies the delay differential equation \eqref{eq:Rdeter}, this gives us,
     \begin{align}
         \int_{-\infty}^r R'(r-s) dW_s^H  & = b  \int_{-\infty}^r R(r-s)  dW_s^H  + \sum_{j=1}^{N} c_j \int_{-\infty}^r R(r-s-\tau_j)  dW_s^H \nonumber\\
         & =  \frac{b}{\sigma} X_r^{\stat} + \sum_{j=1}^{N} c_j \int_{-\infty}^{r-\tau_j} R(r-s-\tau_j)  dW_s^H \nonumber\\
         & = \frac{b}{\sigma} X_r^{\stat}  + \sum_{j=1}^{N} \frac{c_j}{\sigma}  X_{r-\tau_j}^{\stat}\label{eq:DiffXstat5}
     \end{align} Substituting equation \ref{eq:DiffXstat5} into equation \eqref{eq:DiffXstat4}, we arrive at
             \begin{align*}
         X_t^{\stat} - X_u^{\stat} 
         & =   \int_u^t  b X_r^{\stat}  + \sum_{j=1}^N c_j  X_{r-\tau_j}^{\stat} dr + \sigma \left( W_t^H - W_u^H\right),
     \end{align*} 
     as required.
\end{proof}

\begin{corollary}\label{cor:Stat}
    The process $Y^{\stat}$ satisfies the stochastic differential equation \eqref{eq:Y}.
\end{corollary}
\begin{proof} Applying Proposition \ref{Prop:Xstat}, we have that
\begin{align*}
 Y_t^{\stat} - Y_u^{\stat}  & =  X_t^{\stat} - X_u^{\stat} \\
         & =  \int_u^t\left(  b (X_r^{\stat} +\mu) + \sum_{j=1}^N c_j  (X_{r-\tau_j}^{\stat}+\mu)  -\mu\left(b + \sum_{j=1}^N c_j \right) \right) dr\\
         & \qquad + \sigma \left( W_t^H - W_u^H\right)\\
         & = \int_u^t\left( a +   b Y_r^{\stat} + \sum_{j=1}^N c_j  Y_{r-\tau_j}^{\stat} \right) dr + \sigma \left( W_t^H - W_u^H\right),
\end{align*} where in the last equality we used the fact that $a = -\mu \left(b +\sum_{j=1}^N c_j\right)$.
\end{proof}

\subsection{Uniqueness} 

Now we will prove that $Y^{\stat}$ is the unique stationary solution of \eqref{eq:Y}. To this end, we first establish uniqueness in law and then show that there is a unique stationary solution.

\begin{proposition}\label{eq:proplaw}
Any two stationary solutions of the stochastic differential equation \eqref{eq:Y} have the same Gaussian distribution.
\end{proposition}
\begin{proof}
Let us assume that $Y^{(1)} = \left(Y_t^{(1)}\right)_{t\in \mathds{R}}$ and $Y^{(2)} = \left(Y_t^{(2)}\right)_{t\in \mathds{R}}$ are two stationary solutions of \eqref{eq:Y}. This means that
    \[ Y_t^{(k)} -Y_u^{(k)} = \int_{u}^t\left( a +b Y_s^{(k)} + \sum_{j=1}^N c_j Y_{s-\tau_j}^{(k)} \right)ds + \sigma (W_t^H -W_u^H)\]
     for $k=1,2$ and $t>u \in \mathds{R}$. Let us define the process $Z=(Z_t)_{t\in \mathds{R}}$ as
\[ Z_t = Y_t^{(1)} - Y_t^{(2)} \text{ for } t\in \mathds{R}.\] From the definitions of $Y^{(1)}$ and $Y^{(2)}$, we arrive at
\[Z_t - Z_u =  \int_{u}^t\left( b Z_s +\sum_{j=1}^N c_j Z_{s-\tau_j} \right)ds. \] For any fixed $u\in \mathds{R}$, the process $Z$ satisfies the initial value problem
\begin{align}
    dZ_t & = \left( bZ_t + \sum_{j=1}^N c_j  Z_{t-\tau_j} \right) dt \text{ for } t > u,\nonumber \\
    Z_t &=  Y_t^{(1)} - Y^{(2)}_t \text{ for }  t \in [u-\tau_N, u], \label{IVP_stat}
\end{align}  From Proposition \ref{prop:DeterDelay}, we have that $Z$ satisfies
\[ Z_t = Z_u R(t-u) +  \sum_{j=1}^N c_j \int_{-\tau_j}^0 R(t-\tau_j - s-u) Z_{s+u} ds \text{ almost surely. }  \] Due to the fact that \(W^H\) has continuous sample paths and the drift term is given by a
Lebesgue integral, every solution of \(\eqref{eq:Y}\) admits continuous sample paths. Therefore \(Z=Y^{(1)}-Y^{(2)}\) also has continuous sample paths. Hence, for
fixed \(u\in\mathds R\),
\begin{equation}
    \sup_{s\in[u-\tau_N,u]}|Z_s|<\infty
\text{ almost surely}.\label{eq:intialbound}
\end{equation}

Since the parameters satisfy condition \eqref{eq:cond1}, we have the exponential bound~\eqref{eq:Rbound}, and hence by \eqref{eq:intialbound} the process $Z$ satisfies
\[ |Z_t| \to 0 \text{ as } t\to \infty \text{  almost surely.}  \]

Let us now fix the natural number $m\in\mathds{N}$, and real numbers $s_1,\ldots,s_m \in \mathds{R}$. Since $Y^{(1)}$ and $Y^{(2)}$ are stationary, we have that
\begin{equation}\label{eq:StationaryProperty}
    \left(Y^{(k)}_{s_1},\ldots, Y^{(k)}_{s_m}\right) \stackrel{d}{=} \left(Y^{(k)}_{t+s_1},\ldots, Y^{(k)}_{t+s_m}\right),  \text{ for every } t\in \mathds{R} \text{ and } k=1,2.  
\end{equation} For every $\ell \in \{1,2,\ldots, m\}$ we have that
\[ Y^{(1)}_{s_\ell + t} - Y^{(2)}_{s_\ell+t}  =Z_{s_{\ell}+t} \to 0 \text{ almost surely as } t\to \infty.\]
Hence we have that
\begin{equation}\label{eq:limit0}
    \left(Y^{(1)}_{s_1+t},\ldots, Y^{(1)}_{s_m+t}\right)  - \left(Y^{(2)}_{s_1+t},\ldots, Y^{(2)}_{s_m+t}\right) \to \bm{0} \text{ almost surely } \text{ as } t\to \infty.
\end{equation} Using the result in equation \eqref{eq:limit0} and the dominated convergence theorem, we obtain that
\[E\left[\Phi \left(Y^{(1)}_{s_1+t},\ldots, Y^{(1)}_{s_m+t}\right)  - \Phi\left(Y^{(2)}_{s_1+t},\ldots, Y^{(2)}_{s_m+t}\right)   \right]\to 0  \text{ as } t\to \infty, \] for every $\Phi:\mathds{R}^m \to \mathds{R}$ bounded continuous function. From the stationarity property in \eqref{eq:StationaryProperty}, we have shown that
\[ E\left[\Phi \left(Y^{(1)}_{s_1},\ldots, Y^{(1)}_{s_m}\right)  - \Phi\left(Y^{(2)}_{s_1},\ldots, Y^{(2)}_{s_m}\right)   \right] =0,\] which implies that
\[ \left(Y^{(1)}_{s_1},\ldots, Y^{(1)}_{s_m}\right) \stackrel{d}{=} \left(Y^{(2)}_{s_1},\ldots, Y^{(2)}_{s_m}\right).\] We have just proved that $Y^{(1)}$ and $Y^{(2)}$ have the same finite-dimensional distribution.

Finally, we know that $Y^{\stat}$ is a stationary solution of equation \eqref{eq:Y}; see Corollary \ref{cor:Stat}. Due to the fact that $Y^{\stat}$ is a  Gaussian process, and  the uniqueness in law we arrive at the desired result.
\end{proof}

Finally, we will prove that $Y^{\stat}$ is the unique almost surely stationary solution. 
\begin{proposition}
Under condition \eqref{eq:cond1}, the process $Y^{\stat}$ defined in \eqref{eq:StatSol} is the unique  stationary solution of \eqref{eq:Y}.
\end{proposition}
\begin{proof}
As in the proof of Proposition \ref{eq:proplaw}, let us assume that $Y_t^{(1)}$ and $Y_t^{(2)}$ are two stationary solutions of equation \eqref{eq:Y}. As we did before, we define the process $Z=(Z_t)_{t\in \mathds{R}}$ as
\begin{equation*}
    Z_t = Y_t^{(1)} - Y_t^{(2)} \text{ for } t\in \mathds{R}.
\end{equation*} As we have shown in the proof of Proposition \ref{eq:proplaw}, for a fixed $u\in \mathds{R}$ the process $Z$ can be written as
\begin{equation}
 Z_t = Z_u R(t-u) +  \sum_{j=1}^N c_j \int_{-\tau_j}^0 R(t-\tau_j - s-u) Z_{s+u} ds \text{ almost surely. }    \label{eq:Zsol}
\end{equation}
By Proposition \ref{eq:proplaw}, we have that \(Y^{(1)}\) and \(Y^{(2)}\) have the same Gaussian distribution with finite variance, so
\begin{equation}
\sup_{s\in\mathds{R}} {E}\!\left[\left|Z_s\right|^2\right]
\leq
2\sup_{s\in\mathds{R}} {E}\!\left[\left|Y_s^{(1)}\right|^2\right]
+
2\sup_{s\in\mathds{R}} {E}\!\left[\left|Y_s^{(2)}\right|^2\right]
< \infty. \label{eq:inneUnique}
\end{equation}

 For a fixed $t>0$, using the result in  \eqref{eq:Zsol}, we arrive at
\begin{align}
    \left|E\left[Z_t^2\right]\right| 
    & \leq E\left[ |Z_u|^2\right]  |R(t-u)|^2 +  E\left[\left(\sum_{j=1}^N |c_j| \int_{-\tau_j}^0 |R(t-\tau_j - s-u)| |Z_{s+u}| ds\right)^2\right]  \nonumber\\
    & + 2  E\left[ |Z_u|^2\right]^{1/2}  |R(t-u)|   E\left[\left(\sum_{j=1}^N |c_j| \int_{-\tau_j}^0 |R(t-\tau_j - s-u)| |Z_{s+u}| ds\right)^2\right]^{1/2}  \nonumber \\
   & \leq \sup_{s\in\mathds{R}} E\left[ |Z_s|^2\right] \left(    |R(t-u)|^2 + \left(\sum_{j=1}^N |c_j| \int_{-\tau_j}^0 |R(t-\tau_j - s-u)| ds\right)^2 \right.  \nonumber\\
   & \left. + 2   |R(t-u)|   \left(\sum_{j=1}^N |c_j| \int_{-\tau_j}^0 |R(t-\tau_j - s-u)|  ds\right) \right) \label{eq:inneuni2}
\end{align} Taking the limit as $u\to -\infty$ on both sides of the inequality \eqref{eq:inneuni2},
we arrive at
$  \left|E\left[Z_t^2\right]\right| = 0,$ where we used both \eqref{eq:cond1} and \eqref{eq:inneUnique}. From Proposition \ref{eq:proplaw}, we then have that
$E[Z_t] = 0.$ We have just proved that
\[ Z_t = 0 \text{ almost surely for all } t\in \mathds{R}.\] This implies that for a fixed $t\in \mathds{R}$, there exists an event $A_t \in \mathcal{F}$ with probability one such that
\[ A_t = \{\omega \in \Omega:   Z_t(\omega) = 0  \}.\] We will prove that the event 
\[ A_{\mathds{Q}} = \mathcal{C} \cap \bigcap_{q\in \mathds{Q}} A_q\] has probability one, where
\[ \mathcal{C} = \{\omega \in \Omega: t\to Z_t(\omega) \text{ is continuous on } \mathds{R} \}. \]  For $k=1,2$, due to the fact that $Y^{(k)}$ is the stationary solution of \eqref{eq:Y}, it has almost surely continuous sample paths. So, the process $Z$ also has  continuous sample paths with probability one and hence $\mathds{P}(\mathcal{C})=1$. Since for each  $q\in \mathds{Q}$ the event $A_q$ has probability one, the event $A_{\mathds{Q}}$ has probability one because it is a countable intersection of events all with probability one.

Finally, fix any element $\omega\in A_{\mathds{Q}}$. Using the fact that $\mathds{Q}$ is dense in $\mathds{R}$, for each $t\in \mathds{R}$ there exists a sequence $(q_n)_{n=1}^\infty \subset \mathds{Q}$ such that $q_n \to t$ as $n\to \infty$. By continuity of $Z$ we have that
\[Z_t (\omega) = \lim_{n\to \infty} Z_{q_{n}}(\omega) =0,\] where the last equality comes from the fact that $\omega \in A_\mathds{Q}$. So we have proved that
\[ \{Z_{t}(\omega) = 0 \text{ for all } t\in \mathds{R} \} \text{ for every } \omega\in A_{\mathds{Q}}.\] This implies that
\begin{align*}
  \mathds{P}\left( Y_{t}^{(1)}= Y_{t}^{(2)} \text{ for all } t\in \mathds{R} \right)= \mathds{P}\left( Z_{t}= 0 \text{ for all } t\in \mathds{R} \right)=1,
\end{align*} which proves the desired result.
\end{proof}

\section{Conclusion}

In this paper, we studied a linear stochastic delay differential equation driven by fractional Brownian motion and involving a finite number of discrete delays. The model combines two complementary sources of memory: the dependence generated by the fractional Brownian noise and the delayed feedback generated by the drift. By using the fundamental solution of the associated deterministic delay equation, we obtained an explicit representation of the solution in Section~\ref{sect:model}. This explicit solution allowed us to derive formulas for the mean and auto-covariance function of the process as well as obtaining a closed formula for the limiting distribution in Section \ref{sect:limiting}. Finally in Section \ref{sect:Stationary}, by considering a two-sided fractional Brownian motion, we constructed a strictly stationary solution on the whole real line and proved uniqueness of the stationary solution. This gives a complete description of the stationary regime of the model and connects the time-domain covariance analysis with a spectral representation.

Several directions remain open for future research. A first natural extension is the development of statistical estimation methods for the parameters of the model. Estimation of the drift coefficients, volatility coefficient, delay parameters, and Hurst parameter would be essential for applications. This could be approached by adapting methods developed for fractional Ornstein--Uhlenbeck processes, such as those in \citet{Kleptsyna2002} and \citet{HU20101030}, as well as inference techniques for affine stochastic delay equations, such as \citet{Kulcher2013}. In the specific context of fractional stochastic delay equations, the estimation problem is also related to the work of \citet{Rao2008}.

A second direction is the numerical analysis of the model. Although the explicit solution is useful theoretically, simulation and estimation require efficient numerical methods for the deterministic fundamental solution and for the fractional stochastic convolution. This becomes especially important when the number of delays is large or when the delays are close to each other. Future work could investigate stable discretization schemes, fast convolution methods, and simulation algorithms inspired by the hybrid schemes used for Brownian semistationary and rough-volatility-type processes; see, for example, \citet{Bennedsen2017}. Such numerical methods would also make it possible to study the finite-sample behaviour of estimators and to compare the model with empirical data.

A third possible extension concerns more general delay structures. The present paper focuses on finitely many discrete delays. In applications, however, delayed feedback may be distributed over a time interval rather than concentrated at finitely many lags. Extending the analysis to distributed delays, state-dependent delays, or random delays would provide a more flexible modelling framework. The stability theory for functional differential equations developed by \citet{hale2006functional} may provide useful tools for this purpose. Related developments for stochastic delay equations driven by fractional Brownian motion can be found in \citet{Ferrante2006} and, more recently, in \citet{ghani2025singular}.

Finally, the model can be further developed for applications in finance and energy. In finance, the process studied in this paper could be used as a Gaussian factor in stochastic volatility models with delayed feedback, combining ideas from rough volatility \citep{gatheral2022volatility} with stochastic delay volatility models such as \citet{GuineaJulia2024}. In energy markets, delay and memory effects naturally arise in the modelling of production, demand, and weather-driven quantities, connecting the present framework with continuous-time stationary models for wind power such as \citet{rohde2019continuous} and with fractionally filtered delay models such as \citet{Davis2020}. These applications suggest that fractional delay models may provide a useful bridge between tractable Gaussian modelling, long-memory behaviour, and delayed feedback mechanisms observed in real data.

\bibliography{ref}
\appendix

\section{\protect$M$-operator and Laplace transform of $R$}\label{sect:laplace}

In this appendix, we recall the construction of the \(M\)-operator used in the white-noise approach to fractional Brownian motion. Fractional white-noise theory was developed by \citet{ElliottVanDerHoek2003} and \citet{BiaginiOksendalSulemWallner2004}. We summarize the definitions and properties following \citet[Chapter 4]{biagini2008stochastic}.

Let $\mathcal{S}(\mathds{R})$ denote the Schwartz space of rapidly decreasing
smooth functions on $\mathds{R}$, and let $\mathcal{S}'(\mathds{R})$ be its
dual. We can define a probability space $(\Omega, \mathcal{F},\mathds{P})$, where $\Omega:=\mathcal{S}'(\mathds{R})$, $\mathcal{F}:=\mathcal{B}(\mathcal{S}'(\mathds{R}))$  and the probability measure $\mathds{P}$ is defined as the measure satisfying
\begin{equation*}
\int_{\mathcal{S}'(\mathds{R})} \exp(i\langle\omega,f\rangle)\,d\mathds{P}(\omega)
=
\exp\left(-\frac{1}{2}\|f\|_{L^2(\mathds{R})}^2\right)
\qquad
\text{for all }f\in\mathcal{S}(\mathds{R}),
\end{equation*} where  $\langle\omega,f\rangle=\omega(f)$ is the action of
$\omega\in\Omega=\mathcal{S}'(\mathds{R})$ on $f\in\mathcal{S}(\mathds{R})$. The measure $\mathds{P}$ is called the white noise probability measure. Since \(S(\mathds R)\) is dense in \(L^2(\mathds R)\), the action 
\(\langle \omega,f\rangle\) can be extended  to every \(f\in L^2(\mathds R)\).
In particular, we can define the standard Brownian motion $\widetilde{W}=(\widetilde{W}_t)_{t\in \mathds{R}}$ in a distributional sense as
\[ \widetilde{W}_t := \left\langle \omega,\mathds{I}_{[0,t]}(.) \right\rangle \text{ for } t\in \mathds{R},\] where
\[
\mathds{I}_{[0,t]}(s)
=
\begin{cases}
1, & 0\le s\le t,\\
-1, & t\le s\le 0,\quad \text{except }t=s=0,\\
0, & \text{otherwise}.
\end{cases}
\]

By the Kolmogorov continuity criterion, we can find a modification of $\widetilde{W}$ with continuous sample paths, which we denote $W = (W_t)_{t\in \mathds{R}}$. In fact, the Wiener integral with respect to the standard Brownian motion can be represented as
\begin{equation}\label{eq:WhiteNoiseInt}
    \int_{\mathds{R}} f(t) dW_t = \langle \omega ,f \rangle \text{ for } f\in L^2(\mathds{R}).
\end{equation}

It is possible to define fractional Brownian motion using the white noise approach, but to do so, we need the $M$-operator.
\begin{definition}\label{def:MOperator} \citep[p.~304]{ElliottVanDerHoek2003}
    Let $0<H<1$. The $M$-operator  is defined on a function $f\in \mathcal{S}(\mathds{R})$ by
    \[\widehat{Mf}(y) = \sqrt{\kappa_H} |y|^{1/2-H} \widehat{f}(y), \text{ for } y\in \mathds{R},\]
    where
    \[\widehat{g}(y) = \int_{\mathds{R}} e^{- ix y} g(x) dx\text{, and }  \kappa_H = \sin(\pi H) \Gamma(2H+1).\]
\end{definition} The definition of the $M$-operator can be extended to the  space
\[
L_H^2(\mathds{R})
:=\left\{
f:\mathds{R}\to\mathds{R}\ : \ |y|^{1/2-H}\hat{f}(y)\in L^2(\mathds{R})
\right\}.
\] 
As in the definition of the standard Brownian motion, we can define the fractional Brownian motion $ \widetilde{W}^H = \left(\widetilde{W}^H_t\right)_{t\in \mathds{R}} $ using the $M$-operator,
\[ \widetilde{W}^H_t  :=  \langle \omega , M\mathds{I}_{[0,t]}(.) \rangle \text{ for } t\in \mathds{R}.\] 
The result in \eqref{eq:WhiteNoiseInt} allows us to express fractional Brownian motion in terms of standard Brownian motion and the $M$-operator as
\begin{equation}\label{eq:WhiteNoiseFract}
\widetilde{W}_t^H = \langle \omega , M\mathds{I}_{[0,t]}(.) \rangle = \int_{\mathds{R}} M \mathds{I}_{[0,t]}(s) dW_s(\omega).  
\end{equation}

By the Kolmogorov continuity criterion, we can find a continuous modification of $\widetilde{W}^H$, which we denote by $W^H = \left(W^H_t\right)_{t\in \mathds{R}}$.  In fact, it can be shown that the Wiener integral with respect to the fractional Brownian motion $W^H$ can be expressed in terms of the $M$-operator and the standard Brownian motion $W$,
\begin{equation}\label{eq:FracIntMop}
    \int_{\mathds{R}} f(t) dW_t^H(\omega) =  \langle \omega , M f(.) \rangle = \int_{\mathds{R}} M f(t) dW_t(\omega),
\end{equation} for every $f\in L^2_H(\mathds{R}).$

\begin{remark}
The Fourier multiplier in Definition \ref{def:MOperator} contains the normalization factor $\sqrt{\kappa_H}$, where
\[
\kappa_H = \sin(\pi H)\Gamma(2H+1).
\]\citet{ElliottVanDerHoek2003} use the unnormalized multiplier
\[
|\lambda|^{1/2-H}.
\] With that convention, the definition of the fractional Brownian motion given in \eqref{eq:WhiteNoiseFract} satisfies
\begin{align*}
E\!\left[\widetilde{W}_t^H \widetilde{W}_s^H\right] & =   \int_{\mathds{R}} M\mathds{I}_{[0,t]}(x)\,M\mathds{I}_{[0,s]}(x) \,dx\\
& = \frac{1}{2\sin(\pi H)\Gamma(2H+1)}\left(|t|^{2H} + |s|^{2H} - |t-s|^{2H} \right)
\end{align*} \citep[p.~304]{ElliottVanDerHoek2003}. 
Therefore, to obtain the standard fractional Brownian motion used throughout this paper, we need to include the factor $\sqrt{\kappa_H}$ in the definition of $M_H$.
\end{remark}

The $M$-operator is used in Proposition \ref{prop:stat} to calculate the auto-covariance of the process $Y^{\stat}$ in \eqref{eq:StatSol}. To that end, we need to compute the Fourier transform of the function $R$ in \eqref{eq:R}, and to prove that $R\in L_H^2(\mathds{R})$. 

First, we compute the Laplace transform  of the function $R$. For suitable $f:[0,\infty)\rightarrow\mathds{C}$  and $s\in\mathds{C}$, the Laplace transform is defined as
 \[
  L_f(s)  = \int_0^\infty f(u) e^{-s u} du.
 \]  By properties of the Laplace transform \citep[see, for example][Theorems 2.1, 2.4, Appendix B]{Dyke2014}, equation \eqref{eq:Rdeter} can be written as
\begin{align}
       sL_R(s) - R(0)
       & =  bL_R(s) +  \sum_{j=1}^N c_je^{-\tau_j s} L_R(s).  \label{eq:deter3}
\end{align} Rearranging equation \eqref{eq:deter3}, we obtain 
\begin{align}
    L_R(s) & = \frac{1}{s-b-\sum_{j=1}^N c_j e^{-\tau_j s}}.
   \label{eq:Laplace_end}
\end{align} Hence, the Fourier transform of $R$ can be computed from \eqref{eq:Laplace_end} as:
\[ \widehat{R}(u) =  L_R(i u) \text{ for } u\in \mathds{R}.\]

Finally, we conclude this appendix with the following result.

\begin{proposition}\label{prop:SpaceL2h}
    If condition \eqref{eq:cond1} is satisfied, then $R\in L^2_H(\mathds{R})$. 
\end{proposition}
\begin{proof}
Applying the result in  \eqref{eq:Rbound},  we bound the Fourier transform of $R$ as 
 \begin{align}
     \left| \widehat{R}(u) \right| & \leq \int_{\mathds{R}}  \left|e^{iu x} \right| \left| R(x)\right| dx \nonumber\\
     & = \int_{\mathds{R}}  \left|e^{iu x} \right| \left| R(x) \mathds{1}_{[0,\infty)}(x)\right| dx \nonumber \\
     &\leq  M \int_0^\infty  e^{-\lambda x} dx  = \frac{M}{\lambda},\label{eq:hatRbound}
 \end{align} where $M,\lambda>0$.
To obtain the desired result, we need to prove that
\[ \int_{\mathds{R}} |u|^{1-2H}  \left| \widehat{R}(u) \right|^2 du <\infty.\]
To prove this result, we split the previous integral into two parts,
\begin{equation} \label{eq:IntSplit}
    \int_{\mathds{R}} |u|^{1-2H}  \left| \widehat{R}(u) \right|^2 du  =  \int_{[-A,A]} |u|^{1-2H}  \left| \widehat{R}(u) \right|^2 du +  \int_{[-A,A]^c} |u|^{1-2H}  \left| \widehat{R}(u) \right|^2 du,
\end{equation} for some $A> 2\left(|b| + \sum_{j=1}^N |c_j|\right)$. The bound in \eqref{eq:hatRbound} allows us to bound the first integral in \eqref{eq:IntSplit},
\begin{equation}
     \int_{[-A,A]} |u|^{1-2H}  \left| \widehat{R}(u) \right|^2 du \leq \frac{M^2}{\lambda^2} \int_{[-A,A]} |u|^{1-2H} du <\infty,\label{eq:firstIntBound}
\end{equation} where the last inequality follows from the fact that $1-2H>-1$, since $H\in (0,1)$. We now bound the second integral in \eqref{eq:IntSplit}. By application of the reverse triangle inequality, we have

\begin{align*}
    \left| \frac{1}{\widehat{R}(u)}  \right| & = \left| iu -b - \sum_{j=1}^N c_j e^{-iu \tau_j }\right|
    \geq |u|  -  |b| - \sum_{j=1}^N |c_j|.
\end{align*} If we assume that $u\in [-A,A]^c$, we obtain that
\begin{equation}
      \left| \frac{1}{\widehat{R}(u)}  \right|   \geq  |u|-\frac{A}{2} > |u| - \frac{|u|}{2} = \frac{|u|}{2}.\label{eq:invhatRbound}
\end{equation}

Hence, using the bound in \eqref{eq:invhatRbound}, we can show that the second integral in \eqref{eq:IntSplit} satisfies
\begin{equation}
     \int_{[-A,A]^c} |u|^{1-2H}  \left| \widehat{R}(u) \right|^2 du < 4 \int_{[-A,A]^c} |u|^{-1-2H} du <\infty,\label{eq:secondIntBound}
\end{equation} where the last inequality follows because $H\in (0,1)$. The result follows from \eqref{eq:firstIntBound} and \eqref{eq:secondIntBound}.
\end{proof}

\end{document}